\documentclass[12pt]{amsproc}
\usepackage[utf8]{inputenc}
\usepackage[english]{babel}
\usepackage{amsmath}
\usepackage{amsfonts}
\usepackage{amssymb}

\newtheorem{theorem}{\bf Theorem}[section]
\newtheorem{lemma}[theorem]{\bf Lemma}

\newtheorem{corollary}[theorem]{\bf Corollary}

\title{ On residually finite groups satisfying an Engel type identity} \thanks{Dedicated to Pavel Shumyatsky on the occasion of his 60th birthday}
\author[D.  Silveira ]{Danilo Silveira}
\address{Department of Mathematics, Federal University of Goi\'as, 75704-020 Catal\~ao GO, Brazil}
\email{sancaodanilo@gmail.com}


\begin{document}
	
	\maketitle
	
	\begin{abstract}
		Let $ n, q $ be positive integers. We show that if $ G $ is a finitely generated residually finite group satisfying the identity $ [x,_ny^q]\equiv 1, $ then there exists a function $ f(n) $ such that $ G $ has a nilpotent subgroup of finite index of class at most $ f(n) $.   We also extend this result to locally graded groups.		
	\end{abstract}

	\section{Introduction}
	\label{intro}
	
	Let $ n $ be a positive integer. We say that a group $G$  is (left) $ n $-Engel if it satisfies the identity $[y,{}_n\,x]\equiv 1$, where the word $ [x,_n y] $ is  defined inductively by the rules

	\begin{center}
		$ [x,_1y]=x^{-1}y^{-1}xy,\ \ \  [x,_ny]=[[x,_{n-1}y],y ]  $\ \ \ for all $ n\geq 2. $
	\end{center}

	A important theorem of Wilson \cite[Theorem 2]{W} says that finitely generated  residually finite $ n $-Engel  groups are nilpotent. More specific properties of  residually  finite $n$-Engel groups can be found for example in a theorem of Burns and Medvedev (quoted below as Theorem \ref{BM-thm}) stating that there exist functions  $ c(n) $ and $ e(n) $ such that any residually finite $ n $-Engel group $ G $ has a nilpotent normal subgroup $ N $ of class at most $ c(n) $ such that the quotient group $ G/N $ has exponent dividing $ e(n) $. The interested reader is referred to the survey \cite{T} and references therein for further results on finite and residually finite Engel groups.
	The purpose of the present article is to provide the proof for the following theorem.
	
	\begin{theorem}\label{BM}
		Let $ G $ be a finitely generated residually finite group satisfying the identity $ [x,_ny^q]\equiv 1. $ Then there exists a function $ f(n) $ such that $ G $ has a nilpotent subgroup of finite index of class at most $ f(n) $.
	\end{theorem}
	
	A group is called locally graded if every non-trivial finitely generated subgroup has a proper subgroup of finite index. The class of locally graded
	groups contains locally (soluble-by-finite) groups as well as residually finite groups. We can extend the  Theorem \ref{BM} to the class
	of locally graded groups.

	\begin{corollary}\label{cor}
		Let $ G $ be a finitely generated locally graded group satisfying the identity $ [x,_ny^q]\equiv 1. $ Then there exists a function $ f(n) $ such that $ G $ has a nilpotent subgroup of finite index  of class at most $ f(n) $.
	\end{corollary}
	
	In the next section we describe the Lie-theoretic machinery that will be used in the proof of Theorem \ref{BM}. The proof of the theorem and of the corollary is given in Section 3.

	\section{About Lie algebras}
	
	Let $L$ be a Lie algebra over a field $K$ and  $X$  a subset of $L$. By a commutator in elements of $X$ we mean any element of $L$ that can be obtained as a Lie product of elements of $X$ with some system of brackets. If $x_1,\ldots,x_k,x, y$ are elements of $L$, we define inductively 
	$$[x_1]=x_1; [x_1,\ldots,x_k]=[[x_1,\ldots,x_{k-1}],x_k]$$
	and 
	$[x,_0y]=x; [x,_my]=[[x,_{m-1}y],y],$ for all positive integers $k,m$.  
	As usual, we say that an element $a\in L$ is ad-nilpotent if there exists a positive integer $n$ such that $[x,_na]=0$ for all $x\in L$. 
	Denote by $F$ the free Lie algebra over $K$ on countably many free generators $x_1,x_2,\ldots$. Let $f=f(x_1,x_2,\ldots,x_n)$ be a non-zero element of $F$. The algebra $L$ is said to satisfy the identity $f \equiv 0$ if $f(l_1,l_2,\ldots,l_n) = 0$ for any $l_1,l_2,\ldots,l_n\in L$.
	
	The next theorem represents the most general form of the Lie-theoretical part of the solution of the Restricted Burnside Problem \cite{Z1,Z0,Z16}. It was announced by Zelmanov in \cite{Z1}. A detailed proof can be found in \cite{Z16}.

	\begin{theorem}\label{1}
		Let $L$ be a Lie algebra over a field and suppose that $L$ satisfies a polynomial identity. If $L$ can be generated by a finite set $X$ such that every commutator in elements of $X$ is ad-nilpotent, then $L$ is nilpotent.
	\end{theorem}
	
	\subsection{Associating a Lie ring to a group}

	Let $G$ be a group. A series of subgroups $$G=G_1\geq G_2\geq\dots\eqno{(*)}$$ is called an $N$-series if it satisfies $[G_i,G_j]\leq G_{i+j}$ for all $i,j\geq 1$.  Obviously any $N$-series is central, i.e. $G_i/G_{i+1}\leq Z(G/G_{i+1})$ for any $i$.  Let $p$ be a prime. An $N$-series is called $N_p$-series if $G_i^p\leq G_{pi}$ for all $i$. Given an $N$-series $(*)$, let $L^*(G)$ be the direct sum of the abelian groups $L_i^*=G_i/G_{i+1}$, written additively. Commutation in $G$ induces a binary operation $[\cdot ,\cdot]$ in $L^*(G)$. For homogeneous elements $xG_{i+1}\in L_i^*,yG_{j+1}\in L_j^*$ the operation is defined by $$[xG_{i+1},yG_{j+1}]=[x,y]G_{i+j+1}\in L_{i+j}^*$$ and extended to arbitrary elements of $L^*(G)$ by linearity. It is easy to check that the operation is well-defined and that $L^*(G)$ with the operations $+$ and $[\cdot,\cdot]$ is a Lie ring. If all quotients $G_i/G_{i+1}$ of an $N$-series $(*)$ have prime exponent $p$ then $L^*(G)$ can be viewed as a Lie algebra over the field with $p$ elements. In the important  case where the series $(*)$ is the $p$-dimension central series (also known under the name of Zassenhaus-Jennings-Lazard series) of $G$ we write $D_i=D_i(G)=\prod_{jp^k\geq i} \gamma_j(G)^{p^k}$ for  the $i$-th term of the series of $G$, $L(G)$ for the corresponding associated Lie algebra over the field with $p$ elements  and  $L_p(G)$ for the subalgebra generated by the first homogeneous component $D_1/D_2$ in $L(G)$. Observe that the  $p$-dimension central series is an $N_p$-series (see \cite[p.\ 250]{Huppert2} for details).

	The nilpotency of $L_p(G)$ has strong influence in the structure of a finitely generated pro-$ p $ group $G$.
	The proof of the following theorem can be found in \cite[1.(k) and 1.(o) in Interlude A]{GA}.

	\begin{theorem}\label{3} 
		Let $G$ be a finitely generated pro-$p$ group. If $L_p(G)$ is nilpotent, then $G$ is $p$-adic analytic.
	\end{theorem}
	
	Let $x\in G$ and let $i=i(x)$ be the largest positive integer such that $x\in D_i$ (here, $ D_i $ is a term of the $ p $-dimensional central series to $ G $). We denote by $\tilde{x}$ the element $xD_{i+1}\in L(G)$.  We now quote two results providing sufficient conditions for $\tilde{x}$ to be ad-nilpotent. The following lemma was established in \cite[p. 131]{la}.
	
	\begin{lemma}\label{lazard-ad} 
		For any $x\in G$ we have $(ad\,{\tilde x})^p=ad\,(\widetilde {x^p})$.  
	\end{lemma}
	
	\begin{corollary}\label{lemma-lazard}
		Let $x$ be an element of a group $G$ for which there exists a positive integer $m$ such that $x^m$ is $ n $-Engel. Then $\tilde{x}$ is ad-nilpotent.
	\end{corollary}
	
	The following theorem is a particular case of a result that  was established by Wilson and Zelmanov in \cite{wize}.
	
	\begin{theorem}\label{identity} 
		Let $G$ be a group satisfying an identity. Then for each prime number $p$ the Lie algebra $L_p(G)$ satisfies a polynomial identity. 
	\end{theorem}
	
	\section{Proof of the main theorem}
	
	The following useful result is a consequence of \cite[Lemma 2.1]{W} (see also \cite[Lemma 3.5]{PS-2017} for details).
	
	\begin{lemma}\label{lemma-Pavel}
		Let $ G $ be a finitely generated residually finite-nilpotent group. For each prime $ p $ let $ J_p $ be the intersection of all normal subgroups of $ G $ of finite $ p $-power index. If $ G/J_p $ has a nilpotent subgroup of finite index of class at most $ c $ for each $ p $, then $ G $ also has a nilpotent subgroup of finite index  of class at most $ c $.
	\end{lemma}
	
	\begin{proof}
		It follows from proof of \cite[Lemma 3.5]{PS-2017} that there exists a finite set of primes $ \pi $ such that $ G $ embeds in the direct product $\prod_{p\in \pi} G/J_{p}. $ We will identify $ G $ with its images in direct product. By hypothesis, for any $ p\in \pi $, $ G/J_{p} $ contains a nilpotent subgroup of finite index $H_p $ with class at most $ c. $ Set $ H=\cap_{p\in \pi} {H_p} $. Thus, $ G\cap H $ has finite index in $ G $ and has nilpotency class at mos $ c $, which completes the proof.
	\end{proof}

	Recall that a group is locally graded if every non-trivial finitely generated subgroup has a proper subgroup of finite index.  Note that the quotient of a locally graded group need not be locally graded,  since free groups  are locally graded  (see \cite[6.1.9]{Rob}), but no finitely  generated infinite simple group is locally graded. However, the following  results give a sufficient conditions for a quotient to be locally graded (see \cite{LMS} for details).

	\begin{lemma} \label{HP}   Let $G$ be a locally graded group and $N$ a normal locally nilpotent subgroup of $G$. Then $G/N$ is locally graded.
	\end{lemma}
	
	Let $ p $ be a prime and $ q $ be a positive integer. 
	A finite $p$-group $ G $ is said to be powerful if and only if $ [G,G]\leq G^p $ for $ p\neq 2 $ (or $[G,G]\leq G^4 $ for $ p=2 $), where  $ G^q $ denotes the subgroup of $ G $ generated by all $ q $th powers. 
	While considering a pro-$ p $ group $ G $ we shall be interested only in closed subgroups. So by the commutator subgroup $ G'=[G,G] $ we mean
	the closed commutator subgroup, $ G^q $  means the closed subgroup generated by the $ q $th powers. Similarly to powerful finite $ p $-groups, we may define the powerful pro-$ p $ groups. 
	For more details we refer the reader to \cite[Chapters 2 and 3 ]{LMS}. In \cite{TA} the following useful result for powerful finite $ n $-Engel $ p $-group was established.
	
	\begin{lemma}\label{AG}
		There exists a function $ s(n) $ such that any powerful finite $ n $-Engel $ p $-group is nilpotent of class at most $ s(n) $.
	\end{lemma}
	
	The proof of Theorem \ref{BM} will  requires the  following lemma.
	
	\begin{lemma}\label{main-powerful}
		Let $ s(n) $ be as in Lemma \ref{AG}. If $ G $ is a finitely generated powerful pro-$ p $ group satisfying the identity $ [x,_ny^q]\equiv 1, $  then $ G^q $ has  nilpotency class at most $ s(n) $.
	\end{lemma}
	\begin{proof}
		Since $ G $ satisfies the identity $ [x,_ny^q]\equiv 1, $ we can deduce from \cite[Corollary 3.5]{GA} that  $ H=G^q=\{x^q\ | \ x\in G\} $ is a powerful  $ n $-Engel pro-$ p $ group. According to \cite[Corollary 3.3]{GA}, $ H $ is the inverse limit of an inverse system of powerful finite
		$ p $-groups $ H_\lambda $. Lemma \ref{AG} implies that any group $ H_\lambda $ has class at most $ s(n) $, and so,  $ H $ has class at most $ s(n) $ as well. Finally, by a result due to Zelmanov \cite[Theorem 1]{Z-torsion} saying that any torsion profinite group is locally finite we get that the quotient group $G/H $ is finite. This completes the proof.
	\end{proof}
	
	The proof of Theorem \ref{BM} will also require the following result, due to Burns and Medvedev \cite{BM}.
	
	\begin{theorem}\label{BM-thm}
		There exist functions $ c(n) $ and $ e(n) $ such that any residually finite $ n $-Engel group $ G $ has a nilpotent normal subgroup $ N $ of class at most $ c(n) $ such that $ G/N $ has exponent dividing $ e(n) $.	
	\end{theorem}
	
	We are now ready to embark on the proof of our main result.
	
	\begin{proof}[{\bf Proof of Theorem \ref{BM}}]
		For any positive integer $ n $ let $ s(n)$ and $c(n) $ be as in Lemma \ref{AG} and Theorem \ref{BM-thm}, respectively. Set $ f(n)=\max\{s(n),c(n)\} $.
		Since $ G $ satisfies the identity $ [x,_ny^q]\equiv 1 $ we can deduce from \cite[Theorem A]{BSTT} that $ H=G^q $ is locally nilpotent. According to Lemma \ref{HP}, $ G/H $ is locally graded. By Zelmanov's solution of the Restricted Burnside Problem \cite{Z1,Z0,Z16}, locally graded groups of finite exponent are locally finite (see for example \cite[Theorem 1]{M}), and so $ G/H $ is finite. Thus $ H $ is finitely generated and so it is nilpotent. 
		
		By Lemma \ref{lemma-Pavel},  we can assume that $ H $ is residually (finite $ p $-group) for some prime $ p $. If $ p $ does not divides $ q $, then $ H $ is finitely generated  residually finite $ n $-Engel group. By Theorem \ref{BM-thm}, $ H $ contains a nilpotent normal subgroup $ N $ of  class at most $ f(n) $ such that the quotient group $ G/N $ has exponent dividing $ e(n) $. Thus, we can see that $ G/N $ is finite. Thereby, in what follows we can assume that $ H $ is residually (finite $p$-group), where $ p $ divides $ q $.
		
		Set $ H=\langle h_1,\ldots,h_t\rangle. $ Let $ L=L_p(H) $ be the Lie algebra associated with the $ p $-dimensional central series of $ H $. 
		Then $L$ is generated by $\tilde{h}_i=h_i D_2$, $i=1,2,\dots,t$. Let $\tilde{h}$ be any Lie-commutator in $\tilde{h}_i$ and $h$ be the group-commutator in $h_i$ having the same system of brackets as $\tilde{h}$. 
		Since for any group commutator $h$ in $h_1\dots,h_t$ we have that $h^q$ is $n$-Engel, Corollary \ref{lemma-lazard} shows that any Lie commutator in $\tilde h_1\dots,\tilde h_t$ is ad-nilpotent. 
		Since $H$ satisfies the identity $ [x,_ny^q]\equiv 1 $, by Theorem \ref{identity}, $L$ satisfies some non-trivial polynomial identity. According to Theorem \ref{1} $L$ is nilpotent. 
		
		Let $\hat{H}$ be the pro-$p$ completion of $H$, that is, the inverse limit of all quotients of $ H $ which are finite $ p $-groups. Notice that $\hat{H}$ is finitely generated, being $ H $ finitely generated. 
		Since the finite $ p $-quotients of $ H $ are the same as the finite $ p $-quotients of $ \hat{H} $ by (a) and (d) of \cite[Proposition 3.2.2]{Zale}, we get that $L_p(\hat{H})=L$. Hence, $L_p(\hat{H})$  is nilpotent and so,  $\hat{H}$ is a $p$-adic analytic group by Theorem \ref{3}.  
		By  \cite[1.(a) and 1.(o) in Interlude A]{GA}, $\hat{H}$ is virtually powerful, that is, $ H $ has a powerful subgroup $ K $ of finite index.
		By Lemma \ref{main-powerful},  $ K^q $ has class at most $ f(n). $ Furthermore, it follows from \cite[Theorem 1]{Z-torsion} that group $K/K^q $ is finite. Finally, since $ H $ is residually-$ p $, it embeds in $ \hat{H}. $ Thus, $H\cap K^q$ is a nilpotent subgroup of finite index in $ G $ of class at most $ f(n) $. This completes the proof.
	\end{proof}

	\begin{proof}[{\bf Proof of Corollary \ref{cor}}]
		Let $ f(n) $ be as in Theorem \ref{BM}. It follows from \cite[Theorem C]{BSTT} that $ H=G^q  $ is locally nilpotent. By Lemma \ref{HP}, $ G/H $ is a locally graded group. By Zelmanov's solution of the Restricted Burnside Problem, locally graded groups of finite exponent are locally finite. Thus, $ G/H $ is finite and so, $ H $ is a finitely generated nilpotent group. Since polycyclic groups are residually finite \cite[5.4.17]{Rob}, we can deduce from Theorem \ref{BM} that $ H $ contains a subgroup of finite index and of class at most $ f(n) $. The proof is  complete.	
	\end{proof}

\end{document}